\documentclass{scrartcl}
\usepackage[numbers]{natbib}

\usepackage[dvipsnames]{xcolor} 

\usepackage[bookmarks=true]{hyperref}

\usepackage{amsmath}
\usepackage{amsthm}
\usepackage{amssymb}
\usepackage{stmaryrd}
\usepackage{bm}

\usepackage{tikz} 
\usepackage{tikz-cd}

\theoremstyle{plain}
\newtheorem{theorem}{Theorem}[section]
\newtheorem{lemma}[theorem]{Lemma}
\newtheorem{proposition}[theorem]{Proposition}
\newtheorem{corollary}[theorem]{Corollary}

\theoremstyle{definition}
\newtheorem{definition}[theorem]{Definition}
\newtheorem{example}[theorem]{Example}
\newtheorem{remark}[theorem]{Remark}

\newcommand{\A}{\mathbb{A}}
\newcommand{\C}{\mathbb{C}}
\newcommand{\F}{\mathbb{F}}
\newcommand{\G}{\mathbb{G}}
\newcommand{\N}{\mathbb{N}}

\newcommand{\Q}{\mathbb{Q}}

\newcommand{\Z}{\mathbb{Z}}

\newcommand{\frakm}{\mathfrak{m}}

\newcommand{\calA}{\mathcal{A}}

\newcommand{\calO}{\mathcal{O}}

\newcommand{\calV}{\mathcal{V}}

\newcommand{\Spec}{\mathrm{Spec}}

\newcommand{\Tr}{\mathrm{Tr}}
\newcommand{\Nm}{\mathrm{N}}
\newcommand{\Hom}{\mathrm{Hom}}
\newcommand{\Gal}{\mathrm{Gal}}

\newcommand{\ext}{\mathrm{ext}}

\DeclareMathOperator{\NP}{\mathrm{NP}}
\DeclareMathOperator{\HP}{\mathrm{HP}}

\newcommand{\et}{\mathrm{\acute{e}t}}
\newcommand{\tame}{\mathrm{tame}}
\newcommand{\id}{\mathrm{id}}

\newcommand{\tr}{\mathrm{tr}}

\newcommand{\Frob}{\mathrm{Frob}}

\mathchardef\mhyphen="2D

\DeclareOldFontCommand{\sf}{\normalfont\sffamily}{\mathsf}

\title{Newton Polygons of Sums on Curves II: Variation in $p$-adic Families}
\author{Joe Kramer-Miller and James Upton}
\date{\today}

\begin{document}

\maketitle

\begin{abstract}
	In this article we study the behavior of Newton polygons
	along $\Z_p$-towers of curves. Fix an ordinary curve $X$ 
	over a finite field $\mathbb{F}_q$ of characteristic $p$.
	By a $\Z_p$-tower $X_\infty/X$ we mean a 
	tower of covers $ \dots \to X_2 \to X_1 \to X$ with $\Gal(X_n/X) \cong \Z/p^n\Z$.
	We show that if the ramification along the tower is sufficiently moderate,
	then the slopes of the Newton polygon of $X_n$ are equidistributed
	in the interval $[0,1]$ as $n$ tends to $\infty$. Under a stronger
	congruence assumption on the ramification invariants, we completely determine 
	the slopes of the Newton polygon of each curve. This is the first result towards `regularity'
	in Newton polygon behavior for $\Z_p$-towers over higher genus curves.
	We also obtain similar results for $\Z_p$-towers twisted by a generic tame character.
	
\end{abstract}

\tableofcontents

\section{Introduction}

\subsection{Motivation}
An important theme in modern number theory is that of $p$-variation---the idea
of interpolating a family of mathematical objects or invariants by a $p$-adic parameter.
For instance, Iwasawa constructed an ``arithmetic'' $p$-adic $L$-function, which interpolates
the $p$-part of class groups of a $\Z_p$-tower over a number field. 
The existence (and finiteness) of this object allows one to deduce the remarkable
Iwasawa class number formula (see e.g. \cite{Washington}). 
A general principle is as follows: if $M$ is a motive over a number field $K$
and $\rho:G_K \to \Z_p$ is a additive character, we may consider a family
of motives $M \otimes \rho_\chi$. Here $\rho_\chi$ denotes the composition $\chi\circ \rho$ and $\chi$ varies over the space $\Hom(\Z_p,\C_p^\times)$ of continuous characters. 
One expects that $p$-adic families of invariants can be constructed that interpolate
classical invariants. 
For example, we may consider a motive $M_f$ corresponding to an eigenform $f$.
The algebraic special values of $L(M_f \otimes \chi,s)^\mathrm{alg}$, where $\chi$ is a Dirichlet character of $p$-power degree,
can be parameterized by a $p$-adic $L$-function (see e.g. \cite{Mazur-Tate-Teitelbaum}). This is somewhat miraculous, as it tells us that $L(M_f\otimes \chi,1)^\mathrm{alg}$
is close to $L(M_f\otimes \chi', 1)^\mathrm{alg}$ $p$-adically if $\chi$ and $\chi'$ are sufficiently close. 

It is natural to ask for similar $p$-adic variational phenomena to occur over function fields. 
Let $X$ be a smooth affine curve over a finite field $\mathbb{F}_q$
with smooth compactification $\overline{X}$. Let $M$ be a motive (we are working $p$-adically so it
makes sense to realize $M$ as an overconvergent $F$-isocrystal on $X$) and let $\rho: \pi_1(X) \to \Z_p$ be a continuous character. Unlike the number field case,
it is relatively easy to construct a $p$-adic object that interpolates the $L$-functions $L(M \otimes \rho_\chi,s)$. 
This object will not be well behaved in general---for instance, the underlying Iwasawa module will no longer be finitely
generated. However, we do expect there to be well-behaved variational properties
when $\rho$ is sufficiently ``nice''. When $M$ is trivial this expectation manifests in a program devised by Daqing Wan in \cite{Wan3}.
Wan identifies three properties that one may expect for nice $\rho$:
\begin{enumerate}
	\item (monodromy stability) We say that $\rho$ is \emph{monodromy stable} if there exist $a,b \in \Q$ such that: for $n \gg 0$ and $\chi$ of finite order $p^n$ we have $\deg(L(\rho_\chi,s))=a+bp^n$.
	\item (slope uniformity) We say that $\chi$ is \emph{slope uniform} if for $\chi$ of finite order $p^n$, the slopes of $\NP_q L(\rho_\chi,s)$ are equidistributed in $[0,1]$ as $n$ tends to infinity.
	
	\item (slope stability) We say that $\chi$ is \emph{slope stable} if 
	there exist $\alpha_1,\dots,\alpha_d \in (0,1)\cap \Q$ and $k> 0$ such that
	for all $n>k$ and all $\chi$ of finite order $p^n$, the slopes of $\NP_q L(\rho_\chi,s)$
	are
	\[K\sqcup \bigsqcup_{i=1}^d \Bigg\{\frac{\alpha_i}{p^{n-k}},\frac{\alpha_i+1}{p^{n-k}}, \dots, \frac{\alpha_i+ p^{n-k}-1}{p^{n-k}}  \Bigg\},\]
	where $K$ is a set of $2g$ elements in $[0,1]$.
\end{enumerate}
These properties first appear in work of Davis-Wan-Xiao \cite{Davis}, who prove that all three are satisfied when
$X=\mathbb{A}_{\F_q}^1$ and $\rho$ is a particularly simple representation. 
The representations where we expect (at least some of) these properties to hold can be broken into two classes. The first
class are $\rho$ that \emph{come from geometry} (e.g. $\rho$ arises from the relative $p$-adic \'etale cohomology
of a fiber-wise ordinary smooth proper fibration $Y \to X$). In \cite{Wan3}, Wan conjectures that 
all three properties hold for $\rho$ coming from geometry. The only progress is due to the first author
in \cite{Kramer-Miller2}, where geometric representations are proven to satisfy a slightly weaker version
of monodromy stability. 
The second class of representations are those with \emph{strictly stable monodromy} (see \S \ref{d:striclystablefields}-\ref{d:strictlystablecurves} below). By work of Kosters and Zhu (see \cite{KZ}), we know that strictly stable towers over $\mathbb{A}_{\F_q}^1$ are slope uniform.

In this article we prove slope uniformity for strictly stable representations over any
ordinary curve. We also prove slope stability for strictly stable representations that
satisfy a stronger monodromy condition. To the best of our knowledge, these are the first
examples of these phenomena for curves other than $\mathbb{A}_{\F_q}^1$ and $\mathbb{G}_m$. 
Our proof uses the local-to-global methods developed in \cite{KramerMiller}, \cite{Kramer-Miller3}, and the previous article \cite{Kramer-Miller-Upton}.

\subsection{Main Results}
	\label{ss:towers of curves}
	
	\begin{definition}
		Let $F$ be a field. A \emph{$\Z_p$-tower} of fields $F_\infty/F$ is a sequence of finite Galois extensions
		\begin{equation}\label{eq:Finfty}
			\cdots	\supseteq	F_2	\supseteq	F_1	\supseteq	F_0	=	F,
		\end{equation}
		together with an identification $\Gal(F_n/F) \cong \Z/p^n\Z$ for all $n \geq 0$.
	\end{definition}
	
	Suppose that $F$ is a local field of characteristic $p$. Let $F_\infty/F$ be a $\Z_p$-tower of local fields. For each $n \geq 0$, let $v_n$ denote the highest ramification break (in upper numbering) of the finite extension $F_n/F$. By local class field theory, we have a lower bound on the growth of this sequence:
	\begin{equation}\label{eq:lower}
		v_{n+1}	\geq	pv_n.
	\end{equation}
		
	\begin{definition}\label{d:striclystablefields}
		We say that $F_\infty/F$ has \emph{strictly stable monodromy} if (\ref{eq:lower}) is an equality for all $n \gg 0$. Equivalently, $F_\infty/F$ has strictly stable monodromy if and only if there exists $\delta \in \Z[\tfrac{1}{p}]$ such that
		\begin{equation*}
			v_n	=	\delta p^{n-1}
		\end{equation*}
		for all $n \gg 0$. In this case, we also say that $F_\infty/F$ has \emph{$\delta$-stable} monodromy.
	\end{definition}
	
	\begin{definition}
		Let $X$ be a curve. A \emph{$\Z_p$-tower} of curves $X_\infty/X$ is a sequence of finite \'etale covers
		\begin{equation*}
			\cdots 	\to X_2	\to X_1 \to X_0 =  X.
		\end{equation*}
		together with an identification $\Gal(X_n/X) = \Z/p^n\Z$ for all $n \geq 0$.
	\end{definition}
	
	Let $X/\F_q$ be a smooth affine curve with smooth compactification $\overline{X}$. Let $S$ denote the complement of $X$ in $\overline{X}$. For each $P \in S$, let $F_P$ denote the local field of $X$ at $P$. If $X_\infty/X$ is a $\Z_p$-tower of curves, then by restriction we obtain a $\Z_p$-tower of local fields $F_{P,\infty}/F_P$.
		
	\begin{definition}\label{d:strictlystablecurves}
		We say that $X_\infty/X$ has \emph{strictly stable monodromy} if $F_{P,\infty}/F_P$ has strictly stable monodromy for all $P \in S$. Given a tuple $\bm\delta = (\delta_P)_{P \in S}$ of positive numbers in $\Z[\tfrac{1}{p}]$, we say that $F_\infty/F$ has \emph{$\bm\delta$-stable} monodromy if $F_{P,\infty}/F_P$ has $\delta_P$-stable monodromy for each $P \in S$.
	\end{definition}
	
	Let $X_\infty/X$ be a $\Z_p$-tower of curves, corresponding to a continuous and surjective map $\rho:\pi_1(X) \to \Z_p$. If $\chi \in \Hom(\Z_p,\C_p^\times)$ is a continuous character, then we may form the composite
	\begin{equation*}
		\rho_\chi:\pi_1(X)	\to	\Z_p	\xrightarrow{\chi}	\C_p^\times.
	\end{equation*}
	Thus we obtain a family of $p$-adic characters of $\pi_1(X)$ parameterized by the $p$-adic space $\Hom(\Z_p,\C_p^\times)$. In this light, it is convenient to view $X_\infty/X$ as a family of characters of $\pi_1(X)$. 
	We can now state our first theorem:

	\begin{theorem}
		\label{theorem: intro stable implies uniformity}
		Assume that $\overline{X}$ is ordinary. If $X_\infty/X$ has strictly stable monodromy, then $X_\infty/X$ is slope uniform.
	\end{theorem}
	
	\subsubsection{Results on Newton Polygons}
	
		We now state our results regarding the interaction of Newton and Hodge polygons of characters. Let $X_\infty/X$ be a $\Z_p$-tower of curves with $\bm\delta$-stable monodromy.
		
		\begin{definition}
			The \emph{Hodge polygon} $\HP (X_\infty/X)$ is the convex polygon with slope set
			\begin{equation*}
				\{ \underbrace{0,...,0}_{g-1+|S|}\} \sqcup \bigsqcup_{P \in S} \left\{	\frac{p-1}{\delta_P},\frac{2(p-1)}{\delta_P},...\right\}.
			\end{equation*}
		\end{definition}
		
		Let $\chi:\Z_p \to \C_p^\times$ be a finite non-trivial character of order $p^{m_\chi}$. We attach to $\chi$ the parameter
		\begin{equation*}
			\pi_\chi = \chi(1) - 1,
		\end{equation*}
		and regard $\chi$ as taking values in the ring $\Z_p[\pi_\chi]$. Thus, the Artin $L$-function $L(\rho_\chi,s)$ is a polynomial with coefficients in $\Z_q[\pi_\chi]$. In this article, it will be convenient for us to adopt the following normalization convention for Newton polgyons:
		
		\begin{definition}
			Let $\pi_{q,\chi} = \pi_\chi^{v_p(q)}$. The \emph{normalized Newton polygon} of $\rho_\chi$ is
			\begin{equation*}
				\NP(\rho_\chi)	=	\NP_{\pi_{q,\chi}} L(\rho_\chi,s).
			\end{equation*}
		\end{definition}
		
		Let $e_\chi = v_{\pi_\chi}(p) = p^{m_\chi-1}(p-1)$ denote the absolute ramification index of $\Z_p[\pi_\chi]$. Let $m_0$ be the smallest non-negative integer such that for each $P \in S$ the highest ramification break (in upper numbering) of $F_{P,m_0+1}/F_P$ is $p^{m_0} \delta_P$.
	From \eqref{eq:lower} we see that the highest ramification break of $F_{P,m}/F_P$ is $p^{m-1} \delta_P$ for all $m> m_0$.		
		By work of the first author in \cite{Kramer-Miller3}, the Hodge polygon $\HP(X_\infty/X)$ provides a common lower bound for $\NP(\rho_\chi)$ whenever $m_\chi > m_0$:
		\begin{equation} \label{eq: Newton over Hodge finite characters}
			\NP^{< e_\chi}(\rho_\chi) \succeq \HP^{< e_\chi}(X_\infty/X).
		\end{equation}
		If $\overline{X}$ is ordinary, then these polygons have the same terminal point. 
		The next theorem greatly improves the estimate \eqref{eq: Newton over Hodge finite characters} when $\overline{X}$ is ordinary. It shows that the Newton and Hodge polygons touch periodically:

		
		\begin{theorem}\label{t:globaltouching}
			Assume that $\overline{X}$ is ordinary. Let $d = p^{m_0}\sum_{P \in S} \delta_P$. For all $m_\chi > m_0$ and all $1 \leq n < p^{m_\chi-m_0-1}$, the polygons $\NP^{< e_\chi}(\rho_\chi)$ and $\HP(X_\infty/X)$ agree on the interval
			\begin{equation}\label{eq:globaltouching}
				[g-1+nd,g-1+|S|+nd].
			\end{equation}
		\end{theorem}
		
		Theorem \ref{theorem: intro stable implies uniformity} is then an easy consequence of Theorem \ref{t:globaltouching}. If we impose additional congruence conditions on the $\delta_P$, we can completely determine the Newton polygon of each finite character:

		\begin{theorem}
		\label{theorem: intro stable with congruence implies stability}
			Let $X_\infty/X$ be a $\Z_p$-tower of curves with $\bm\delta$-stable monodromy. Then $\NP^{< e_\chi} (\rho_\chi) = \HP^{< e_\chi}(X_\infty/X)$ for every finite $\chi$ if and only if:
			\begin{enumerate}
				\item	$\overline{X}$ is ordinary.
				\item	$m_0 = 0$, so that $\delta_P \in \Z$ for all $P \in S$.
				\item	$p \equiv 1 \pmod{\delta_P}$ for all $P \in S$.
			\end{enumerate}
		\end{theorem}
		
		As a consequence of Theorem \ref{theorem: intro stable with congruence implies stability}, we deduce the full form of \cite[Theorem 1.3]{Kramer-Miller-Upton}. 
		This gives a necessary and sufficient condition for the Newton polygon of a single character to coincide with its Hodge bound. Theorem \ref{theorem: intro stable with congruence implies stability} also provides our first examples of slope stability over curves of higher genus:
		
		\begin{corollary}
			Every $\Z_p$-tower of curves $X_\infty/X$ satisfying the conditions of Theorem \ref{theorem: intro stable with congruence implies stability} is slope stable.
		\end{corollary}

	\subsubsection{Towers Twisted by Tame Characters}
	
		Let us continue to assume that $X_\infty/X$ is a $\Z_p$-tower with $\bm\delta$-stable monodromy, corresponding to a continuous and surjective map $\rho:\pi_1(X) \to \Z_p$. Let $\psi:\pi_1(X) \to \Z_p^\times$ be a tame character. We will write $\psi \otimes \rho$ or
		$\psi \otimes X_\infty/X$ for the ``family'' of characters $\psi \otimes \rho_\chi$
		parameterized by the space $\Hom(\Z_p,\C_p^\times)$. 
	By combining Theorem \ref{t:globaltouching} and Theorem \ref{theorem: intro stable with congruence implies stability} with the results of \cite{Kramer-Miller3}, we are able to prove
	stability and uniformity results for this twisted family.
			
		Let $f:X^\tame \to X$ be the finite cyclic covering corresponding to $\ker(\psi)$. For each $P \in S$, let $s_P$ denote the ramification index of $f$ at $P$. If $Q$ is a point of $\overline{X}^\tame$ lying above $P$, then we define $\delta_Q^\tame = s_P \delta_P$. Our results are as follows:
	
	\begin{theorem}\label{t:twisted uniformity}
		Assume $\overline{X}^\tame$ is ordinary. Then $\psi \otimes X_\infty/X$ is slope uniform.
	\end{theorem}
	
	\begin{theorem}\label{t:twisted stability}
		Assume that $\overline{X}^\tame$ is ordinary and that $p \equiv 1 \mod \delta^\tame_Q$ for $Q$ lying above $P \in S$.
		Then $\psi \otimes X_\infty/X$ is slope stable.
	\end{theorem}

	\begin{remark}
		Our assumption that $\overline{X}^\tame$ is ordinary is mild: A Theorem of Bouw \cite{Bouw}
		guarnatees that a generic cover will be ordinary. This uses the fact that $\deg(f)|p-1$ in an essential way. 
		If $\deg(f)$ does not divide $p-1$ the lower Hodge bound is obtained by `averaging' the Hodge decomposition
		of each isotypical component of $H^1(\overline{X}^\tame)$ over the entire Frobenius orbit. 
	\end{remark}

	\begin{remark}
		Our proof of Theorem \ref{t:twisted stability} gives a much stronger result. We completely determine 
		the slopes, which depend on certain monodromy invariants. See \ref{ss:twisted} for more details.
	\end{remark}

	\begin{example}
		Let $\overline{X}=\mathbb{P}_{\F_q}^1$ and let $S=P_1,\dots,P_{2n}$. Let $\psi$ be a quadratic character 
		ramified at each $P_i$. Let $X_\infty/X$ be a $(1)_{P \in S}$-stable tower. Such a tower may be 
		obtained using Artin-Schreier-Witt theory (\S \ref{s:zptow}) from a rational function with simple poles at $P_1,\dots,P_{2n}$. Then $X^\tame$ is a hyperelliptic curve
		of genus $n-1$, which we assume is ordinary. Without being too precise, we obtain the decomposition
		\begin{align*}
			\mathrm{Ind}^{X^\tame}_{X} (X_{\infty}/X) &= X_{\infty}/X \oplus (\psi \otimes X_{\infty}/X).
		\end{align*}
		Let $\chi:\Z_p \to \C_p^\times$ be a finite character and let $\rho_\chi^\tame$ denote the pullback of $\rho_\chi$ along $f$. There is a factorization 
		$L(\rho_\chi^\tame,s)=L(\rho_\chi,s)L(\psi\otimes \rho_\chi,s)$. The slopes of the $L$-functions break up as follows:
		\begin{align*}
			\NP(\rho_\chi^\tame)&=\underbrace{\{0,\dots,0\}}_{3n-2}\sqcup \underbrace{\{e_\chi,\dots,e_\chi\}}_{3n-2}
			\sqcup \bigsqcup_{i=0}^{2n} \Bigg\{\frac{p-1}{2}, 2\frac{(p-1)}{2}, \dots, (2p^{m_\chi-1}-1)\frac{(p-1)}{2} \Bigg\}, \\\
			\NP(\rho_\chi)&=\underbrace{\{0,\dots,0\}}_{2n-1}\sqcup \underbrace{\{e_\chi,\dots,e_\chi\}}_{2n-1}
			\sqcup \bigsqcup_{i=0}^{2n} \Bigg\{2\frac{(p-1)}{2},4\frac{(p-1)}{2},\dots, (2p^{m_\chi-1}-2)\frac{(p-1)}{2} \Bigg\},\\
			\NP(\psi\otimes \rho_\chi)&=\underbrace{\{0,\dots,0 \}}_{n-1}\sqcup \underbrace{\{e_\chi,\dots,e_\chi\}}_{n-1}
			\sqcup \bigsqcup_{i=0}^{2n} \Bigg\{\frac{p-1}{2}, 3\frac{(p-1)}{2}, \dots, (2p^{m_\chi-1}-1)\frac{(p-1)}{2} \Bigg\}.
		\end{align*}
		Let us point out that this decomposition can be deduced from Theorem \ref{theorem: intro stable with congruence implies stability}, without any reference to \cite{Kramer-Miller3}. This is because Theorem \ref{theorem: intro stable with congruence implies stability} allows us to determine $\NP_q(\rho_\chi^\tame)$ and $\NP_q(\rho_\chi)$. The ``leftover'' slopes
		must be the slopes of the quadratic twist $\NP_q(\psi\otimes \chi)$.
	\end{example}
	
	\subsubsection{Equicharacteristic Results}\label{sss:equicharacteristic}
	
		It is well known that the space $\Hom(\Z_p,\C_p)$ may be identified with the open unit disk in $\C_p$ (via the map $\chi \mapsto \pi_\chi$), or equivalently with the $\C_p$-valued points of $\Spec(\Z_p\llbracket  T \rrbracket)$. From this vantage point, it is natural to consider the ``special'' point, which can be thought of as compatifying the open disc. 
		The special point corresponds to an equicharacteristic character $\chi_0:\Z_p \to \F_p\llbracket  T \rrbracket^\times$ defined by
		\begin{equation*}
			\chi_0(a)	=	(1+T)^a \in 1+T\F_p\llbracket  T \rrbracket.
		\end{equation*}
		As before, we define the composite character.
		\begin{equation*}
			\rho_0:\pi_1(X)	\to	\Z_p	\xrightarrow{\chi_0}	\F_p\llbracket  T\rrbracket^\times.
		\end{equation*}
		We may think of $\rho_0$ as the `limit' of characters tending towards the boundary of the open unit disc. 
		\begin{definition}
			The \emph{equicharacterstic $L$-function} of $X_\infty/X$ is the Artin $L$-function $L(\rho_0,s) \in 1+s\F_q\llbracket  T\rrbracket\llbracket  s\rrbracket$
			associated to $\rho_0$.
		\end{definition}
		The equicharacteristic $L$-function describes the limiting behavior of the $L(\rho_\chi,s)$ for finite $\chi$ as $m_\chi \to \infty$. As an example of the well behaved nature of this $L$-function, we have:
		
		\begin{theorem}\label{t:modpentire}
			Suppose that $X_\infty/X$ has strictly stable monodromy. Then $L(\rho_0,s)$ is a $T$-adic entire function on $\F_q(\!(T)\!)$.
		\end{theorem}
		
		In light of Theorem \ref{t:modpentire}, it is convenient for us to put the $L$-functions $L(\rho_0,s)$ and $L(\rho_\chi,s)$ for finite $\chi$ on equal footing. Let $T_q = T^{v_p(q)}$ and define the \emph{normalized Newton polygon} of $\rho_0$ to be the infinite convex polygon
		\begin{equation*}
			\NP(\rho_0)	=	\NP_{T_q} L(\rho_0,s).
		\end{equation*}
		We can now state the equicharacteristic analogues of our main theorems:
		
		\begin{theorem}\label{t:equicharacteristic}
			Assume that $\overline{X}$ is ordinary. Let $X_\infty/X$ be a $\Z_p$-tower with $\bm\delta$-stable monodromy. Then:
			\begin{enumerate}
				\item	The Newton polygon $\NP(\rho_0)$ lies above $\HP(X_\infty/X)$.
				\item	Let $d$ be as in Theorem \ref{t:globaltouching}. For all $n > 0$, $\NP (\rho_0)$ and $\HP(X_\infty/X)$ coincide on the interval (\ref{eq:globaltouching}).
				\item	Suppose that $m_0 = 0$ and that $p \equiv 1 \pmod{\delta_P}$ for all $P \in S$. Then $\NP(\rho_0) = \HP(X_\infty/X)$.	
			\end{enumerate}
		\end{theorem}
		
		\begin{corollary}\label{c:factor}
			In the setting of Theorem \ref{t:modpentire}, the $L$-function $L(\rho_0,s)$ factors over $\F_p(\!(T)\!)$ as a product of polynomials of degree $\leq d$.
		\end{corollary}
		
		\paragraph{The Equicharacteristic Riemann hypothesis}
		
			In light of Corollary \ref{c:factor}, it seems natural to ask if there is a bound on the denominators of the Newton slopes of $L(\rho_0,s)$. This question is in the spirit of Goss's (wide open) Riemann hypothesis
		for Drinfeld modules. Thus we are led to make the following definition:
		
		\begin{definition}
			We say that a $\Z_p$-tower $X_\infty/X$ satisfies the \emph{equicharacteristic Riemann hypothesis} if
			$L(\rho_0,s)$ is a $T$-adic entire function whose zeros are all contained in a finite extension of $\mathbb{F}_p(\!(T)\!)$.
		\end{definition}
		
		It is not clear to us which towers should satisfy the equicharacteristic Riemann hypothesis. A hopeful guess is that
		it holds for all towers with strictly stable monodromy. From
		Theorem \ref{t:equicharacteristic} we can establish the Riemann hypothesis for
		certain towers.
		\begin{theorem}\label{t:equicharacteristicRH}
			Assume that $\overline{X}$ is ordinary. Let $X_\infty/X$ be a $\Z_p$-tower with $\bm\delta$-stable monodromy. Assume one of the two conditions hold:
			\begin{enumerate}
				\item $\delta_P \in \Z$ for all $P \in S$ and $\sum \delta_P < p$.
				\item $\delta_P \in \Z$ for all $P \in S$ and $p \equiv 1\pmod{\delta_P}$ for all $P \in S$.
			\end{enumerate}
			Then the equicharacteristic Riemann hypothesis holds for $X_\infty/X$.
		\end{theorem}
		To deduce the Riemann hypothesis from the first condition, use Corollary \ref{c:factor} in conjunction with the observation that there are finitely many extensions of $\mathbb{F}_q(\!(T)\!)$ with degree less than $p$. To deduce the Riemann hypothesis from the second condition, use the third part of Theorem \ref{t:equicharacteristic}.
	\subsection{Outline of Proof}
	
		\paragraph{Artin-Schreier-Witt Theory and Splitting Functions}
		The purpose of \S \ref{s:zptow} is to upgrade some classical constructions from the theory of exponential sums on tori to more general varieties.
		In \S \ref{ss:asw}-\ref{ss:flatliftings} we recall the basic notions of Artin-Schreier-Witt theory, and give an alternative formulation in terms of the flat liftings used in \cite{Kramer-Miller-Upton}. Our reason for working with flat liftings is that they lead to a very general theory of Dwork splitting functions. Such splitting functions are an integral part of the classical theory of exponential sums (see \cite{Adolphson-Sperber} and \cite{Liu-Wei}), and can be constructed quite explicitly from the Artin-Hasse exponential series (as in \cite{Dwork-zeta-of-hypersurface}). However,
		these classical splitting functions are inapplicable to the study of exponential sums on more general varieties. In \S \ref{ss:split} we construct
		splitting functions using a generalized Artin-Hasse exponential. This generalized
		Artin-Hasse exponential is ``compatible'' with a given lifting of Frobenius (as opposed to
		the classical Artin-Hasse exponential, which is compatible with the Frobenius $t \to t^p$). In particular, it allows us to explicitly define splitting functions over any
		variety.
This approach to splitting functions appears to be genuinely new, and allows us to deduce, for example, that the $F$-crystal associated to \emph{any} character of $\pi_1(X)$ (not necessarily finite) factoring through a $\Z_p$-tower is free. This result is essential in order to apply the main theorems of \cite{Kramer-Miller-Upton}.
			
		\paragraph{Local Towers and Global Extensions}
			In \S \ref{ss:strictlystable} we study the growth conditions satisfied by splitting functions
			associated to a $\delta$-stable tower $F_\infty/F$ of local fields. Every such tower extends in a natural way to a $\delta$-stable tower $X_\infty^\ext/\A_{\F_q}^1$ over the line, whose localization at $\infty$ is $F_\infty/F$. Let $\rho^{\ext}:\pi_q(\A_{\F_q}^1) \to \Z_p$ be the corresponding map.
			The purpose of \S \ref{ss:localtoglobal} is to study the family of characters $\rho_\chi^{\ext}$.
			Using work of Kosters and Zhu in \cite{KZ}, we know that $\HP(X_\infty^{\ext}/\A_{\F_q}^1)$ and $\NP(\rho_\chi^{\ext})$
			touch periodically when $\chi$ is a finite character of sufficiently large order.
			With a minimal amount of extra work, the theory of Kosters-Zhu may be extended to include the equicharacteristic character $\chi_0$ (see Theorem \ref{t:KZ}).
			This periodic touching between Newton and Hodge polygons is necessary to verify a technical hypotheses needed to apply results from \cite{Kramer-Miller-Upton} (see Corollary \ref{c:deltaHodge}). 

		\paragraph{Applying Local-to-Global results}
		We are now in a position to apply the main results from \cite{Kramer-Miller-Upton}. 
		Assume that $\overline{X}$ is ordinary and let $X_\infty/X$ be
		a $\Z_p$-tower with $\bm\delta$-stable monodromy, corresponding to $\rho:\pi_1(X) \to \Z_p$. Let $\chi$ be either
		a finite character or the equicharacteristic character. 
		For each $P \in S$ we may localize to obtain a tower $F_{P,\infty}/F_P$, which extends as above to a $\Z_p$-tower
		tower $X_{P,\infty}^{\ext}/\A_{\F_q}^1$ over the line. 
		
		The main result of \cite{Kramer-Miller-Upton} roughly states that
		$\NP^{<e_\chi}(\rho_\chi)$ and $\HP^{<e_\chi}(X_\infty/X)$ share a vertex
		if and only if $\NP^{<e_\chi}(\rho^{\ext}_{P,\chi})$ and $\HP^{<e_\chi}(X^{\ext}_{P,\infty}/\A_{\F_q}^1)$
		share a corresponding vertex for each $P \in S$. 
		Using the 
		periodic touching between Newton and Hodge polygons from the previous paragraph,
		we obtain periodic touching between $\NP^{<e_\chi}(\rho_{\chi})$ and $\HP^{<e_\chi}(X_{\infty}/X)$.
		This immediately allows us to deduce Theorem \ref{t:globaltouching}. To establish Theorem \ref{theorem: intro stable with congruence implies stability}, we use a classical result on exponential sums (see e.g. \cite{Adolphson-Sperber}). This result says that $\NP^{<e_\chi}(\rho^{\ext}_{P,\chi})$ and $\HP^{<e_\chi}(X^{\ext}_{P,\infty}/\A_{\F_q}^1)$
		are equal when $\chi$ has order $p$ and the congruence conditions of Theorem \ref{theorem: intro stable with congruence implies stability} is satisfied. We then use the work of Kosters-Zhu (and our
		equicharacteristic modification) to see that $\NP^{<e_\chi}(\rho^{\ext}_{P,\chi})$ and $\HP^{<e_\chi}(X^{\ext}_{P,\infty}/\A_{\F_q}^1)$ agree for any finite $\chi$ or for $\chi_0$. 
		Theorem \ref{theorem: intro stable with congruence implies stability}
		then follows from the main result of \cite{Kramer-Miller-Upton}.

	\subsection{Future Work}
		It would be interesting to understand what happens when one twists a $\bm\delta$-stable tower $X_\infty/X$ by a motive $M$ pure of weight $k$. 
		Our expectation is that $M \otimes X_\infty/X$ exhibits slope uniformity in the interval $[0,k+1]$.
		Perhaps a natural first step would be to consider $M$ coming from an Artin representation. Indeed, let $Y \to X$ be a Galois cover
		and consider the pullback tower $Y_\infty/Y$. Then $Y_\infty/Y$ will be $\bm\delta'$-stable for some $\bm\delta'=(\delta'_P)$ that depends
		on $\bm\delta$ and the ramification of $Y\to X$. In particular, $Y_\infty/Y$ is slope uniform by Theorem \ref{theorem: intro stable implies uniformity}. 
		We have a decomposition $\mathrm{Ind}^Y_X(Y_\infty/Y) = \bigoplus \psi \otimes X_\infty/X$, where the sum is taken over the irreducible representations of $\Gal(Y/X)$. 
		We expect that each $\psi\otimes X_\infty/X$ is slope uniform as well. 
		
		Another interesting direction would be to explore other classes of $\Z_p$-towers. Our notion of $\bm\delta$-stable monodromy is natural condition from an analytic standpoint, but is
		somewhat ad-hoc from a geometrical point of view. 
		Instead, one may consider geometric towers, i.e., towers coming from the relative $p$-adic \'etale cohomology of a smooth fibration $Z \to X$. 
		This includes the Igusa tower, where the equicharacteristic $L$-function is closely related to the space of $\F_q\llbracket  T \rrbracket$-valued modular forms.
		In general, geometric towers are not $\bm\delta$-stable. However, as demonstrated by the first author in \cite{Kramer-Miller2}, the
		monodromy of a geometric towers exhibits a similar type of stability. Daqing Wan has conjectured that geometric towers should be slope stable
		in the appropriate sense, but essentially nothing is known in this direction.
		
		Finally, we believe that for a sufficiently well-behaved tower $X_\infty/X$, geometric data should be
		encoded in the various $L$-functions under consideration. For example, we hope to establish a combinatorial formula for the $a$-numbers of each $\overline{X}_n$, depending on the Newton polygon of the equicharacteristic $L$-function and
		the ramification breaks of the tower. This is an ongoing project with Jeremy Booher and Bryden Cais, motivated by computations and conjectures formulated by Booher and Cais in \cite{Booher-Cais}.
		Our hope is that this formula holds for all $\bm\delta$-stable towers over any ordinary curve. It seems likely that the gluing methods
		developed in the prequel article \cite{Kramer-Miller-Upton} will allow one to bootstrap results of this form from the special case $X = \mathbb{A}_{\F_q}^1$. 
		More generally, we hope that for well behaved towers the $T$-adic $L$-function defined Liu-Wan \cite{Liu} should contain information
		on invariants of the Dieudonn\'e module associated to $\mathrm{Jac}(X_n)$.

		\subsection{Acknowledgments}
		We would like to thank Daqing Wan for many discussions about the nature of $\Z_p$-towers.
	
\section{\texorpdfstring{$\Z_p$}{Zp}-Towers}\label{s:zptow}
	
	\subsection{Artin-Schreier-Witt Theory via Witt Vectors}\label{ss:asw}
	
		Let $A$ be a ring and let $W(A)$ denote the ring of $p$-typical Witt vectors of $A$. Recall that as sets we have $W(A) = A^\N$. The ring structure on $W(A)$ is characterized uniquely by the requirement that, for each $i \geq 0$, the \emph{ghost map} $w_i:W(A) \to A$ defined by
		\begin{equation*}
			w_i(a_0,a_1,...) = \sum_{j = 0}^n p^j a_j^{p^{i-j}}
		\end{equation*}
		is a functorial homomorphism of rings (\cite{Haze}, 15.3.10). 
		Let $X$ be an $\F_p$-scheme. Let $X_\et$ denote the small \'etale site of $X$. For each $n \geq 0$, the assignment $U \mapsto W_n(\calO_U)$ defines a sheaf of rings on $X_\et$, which we denote by $\tilde{W}_n$. We define $\tilde{W} = \varprojlim_n \tilde{W}_n$. We equip $\tilde{W}$ with the $\Z_p$-module endomorphism
		\begin{equation*}
			\wp	=	F-\id,
		\end{equation*}
		where $F$ denotes the Frobenius endomorphism on $\tilde{W}$.
		\begin{lemma}\label{l:lang}
			Let $X$ be an $\F_p$-scheme. For each $n \geq 0$, there is an exact sequence of \'etale sheaves of $\Z_p$-modules
			\begin{equation}\label{eq:lang}
				0	\to	W_n(\F_p)	\to	\tilde{W}_n	\xrightarrow{\wp}	\tilde{W}_n	\to	0.
			\end{equation}
		\end{lemma}
		\begin{proof}
			This is well known for the case $n = 0$ (see e.g. \cite{Deligne-sga4.5}). The general case follows from an inductive lifting argument.
		\end{proof}
		
		\begin{lemma}\label{l:Hilbert90}
			Let $X$ be an affine scheme over $\F_p$. Then for all $n \geq 0$, $H^1(X_\et,\tilde{W}_n) = 0$.
		\end{lemma}
		\begin{proof}
			Noting that $\tilde{W}_0 = \G_a$, by Hilbert's theorem 90 \cite[Remark 9.8]{Milne} we have
			\begin{equation*}
				H^1(X_\et,\tilde{W}_0) = H^1(X,\calO_X) = 0,
			\end{equation*}
			since $X$ is affine. The general case follows from an inductive lifting argument as above.
		\end{proof}
		
		\begin{theorem}
			Let $X = \Spec(\overline{A})$ be an affine scheme over $\F_p$. There is a natural isomorphism of $\Z_p$-modules
			\begin{equation*}
				W(\overline{A})/\wp W(\overline{A})	\xrightarrow{\sim}	\Hom(\pi_1(X),\Z_p).
			\end{equation*}
		\end{theorem}
		\begin{proof}
			By Lemma \ref{l:lang}, for each $n \geq 0$ we have a long exact sequence
			\begin{equation*}
				0	\to	W_n(\F_p)	\to	W_n(\overline{A}) \xrightarrow{\wp} W_n(\overline{A})	\to	H^1(X_\et,W_n(\F_p))	\to	H^1(X_\et,\tilde{W}_n)	\to	\cdots.
			\end{equation*}
			Note that $H^1(X_\et,W_n(\F_p)) = \Hom(\pi_1(X),W_n(\F_p))$. By Lemma \ref{l:Hilbert90}, the map
			\begin{equation*}
				W_n(\overline{A})/\wp W_n(\overline{A})	\to	\Hom(\pi_1(X),W_n(\F_p))
			\end{equation*}
			is an isomorphism. The theorem follows by passing to the inverse limit.
		\end{proof}
		
	\subsection{Artin-Schreier-Witt Theory via Flat Liftings}\label{ss:flatliftings}
	
		Let $X=\Spec(\overline{A})$ be an affine scheme of characteristic $p$. It will be convenient to replace the ring of Witt vectors $W(\overline{A})$ by a suitable lifting of $\overline{A}$ to characteristic $0$. Let $R$ be a complete Noetherian local ring with maximal ideal $\frakm$ and residue field $\F_p$. By a \emph{flat lifting} of $X$ over $R$, we will mean a pair $(A,\sigma)$ where $A$ is a lifting of $\overline{A}$ to a flat $R$-algebra, and $\sigma:A \to A$ is a lifting of the absolute Frobenius endomorphism of $X$.
		
		Following our approach in \cite{Kramer-Miller-Upton}, we will typically fix a flat lifting $(A,\sigma)$ of $X$ over $\Z_p$ and then pass to general $R$ by base change. Over $\Z_p$, the connection between flat liftings and Witt vectors is provided by the following:
		
		\begin{lemma}[{\cite[Lemma 17.6.9]{Haze}}]\label{l:ahexp}
			There is a unique ring map $D_\sigma:A \to W(A)$ making the following diagram commute for each $i \geq 0$:
			\begin{equation*}
				\begin{tikzcd}
					A	\arrow[r,"D_\sigma"]	\arrow[dr,"\sigma^i"']	&	W(A)	\arrow[d,"w_i"]	\\
						&	A
				\end{tikzcd}
			\end{equation*}
		\end{lemma}
		
		By Lemma \ref{l:ahexp} and functoriality of the Witt construction, we obtain a \emph{Frobenius-compatible} map
		\begin{equation*}
			\overline{D}_\sigma:A	\xrightarrow{D_\sigma}	W(A)	\to	W(\overline{A}).
		\end{equation*}
		We define as before an additive endomorphism $\wp = \sigma - \id:A \to A$. Let $A^\infty$ denote the $p$-adic completion of $A$. The following justifies our use of flat liftings in the context of Artin-Schreier-Witt theory:
		
		\begin{theorem}\label{t:asw}
			The natural map $A^\infty/\wp A^\infty \to W(\overline{A})/\wp W(\overline{A})$ is an isomorphism. In particular, there is a natural isomorphism of $\Z_p$-modules.
			\begin{equation*}
				A^\infty/\wp A^\infty	\xrightarrow{\sim}	\Hom(\pi_1(X),\Z_p).
			\end{equation*}
		\end{theorem}
		\begin{proof}
			The reduction of this map modulo $p$ is the identity map on $\overline{A}/\wp \overline{A}$. Since both modules are $p$-adically complete, the claim follows from \cite[Lemma 2.3]{Kramer-Miller-Upton}.
		\end{proof}
		
		\begin{remark}
			The isomorphism of Theorem \ref{t:asw} is functorial in the following sense: Let $Y = \Spec(\overline{B})$ be another $\F_p$-scheme, and let $(B,\tau)$ be a flat lifting of $Y$ over $\Z_p$. If $h:A \to B$ is any \emph{Frobenius-compatible} map of flat liftings, then there is a commutative diagram:
			\begin{equation*}
				\begin{tikzcd}
					A^\infty/\wp A^\infty	\arrow[d,"h"']	\arrow[r,"\sim"]	&	\Hom(\pi_1(X),\Z_p)	\arrow[d,"h^*"]	\\
					B^\infty/\wp B^\infty	\arrow[r,"\sim"]	&	\Hom(\pi_1(Y),\Z_p)
				\end{tikzcd}
			\end{equation*}
		\end{remark}
		
		\begin{definition}
			Let $x$ be a closed point of $X$. The \emph{Teichm\"uller lifting} of $x$ is the Frobenius-compatible $\Z_p$-algebra map
			\begin{equation*}
				\hat{x}:A	\xrightarrow{\overline{D}_\sigma}	W(\overline{A})	\to	W(k(x)),
			\end{equation*}
			where $k(x)$ denotes the residue field at $x$.
		\end{definition}
		
		The Teichm\"uller lifting $\hat{x}$ is the unique Frobenius-compatible map lifting the natural map $\overline{A} \to k(x)$. Given an element $f \in A$, we write $f(\hat{x}) \in W(k(x))$ for the value of $f$ at the Teichm\"uller point $\hat{x}$.
		
		
		\begin{proposition}\label{p:aswfrob}
			Let $n \in \N \sqcup \{\infty\}$ and let $f \in A^\infty$. Let $\rho:\pi_1(X) \to \Z_p$ be the associated map from Artin-Schreier-Witt theory. Let $x$ be a closed point of $X$. As elements of $\Z_p$, we have
			\begin{equation*}
				\rho(\Frob_x)	=	\Tr_{k(x)/\F_p}	f(\hat{x}).
			\end{equation*}
			Here, we identify the Galois group of $k(x)/\F_p$ with that of $W(k(x))/\Z_p$.
		\end{proposition}
		\begin{proof}
			Choose a separable closure $k/k(x)$. By Lemma \ref{l:lang}, there exists $b \in W(k)$ such that $\wp(b) = f(\hat{x})$. The isomorphism
			\begin{equation*}
				W(k(x))/\wp W(k(x))	\to	\Hom(\Gal(k/k(x)),\Z_p)
			\end{equation*}
			is explicitly given by $f(\hat{x}) \mapsto (g \mapsto gb-b)$. Let $d = [k(x):\F_p]$. We see that
			\begin{equation*}
				\rho(\Frob_x)	=	F^db-b	=	\sum_{j=0}^{d-1} F^j(Fb-b)	=	\Tr_{k(x)/\F_p}	f(\hat{x}).
			\end{equation*}
		\end{proof}
		
		
	\subsection{Characters and Splitting Functions}\label{ss:split}
	
		In his proof of the rationality of the zeta function, Dwork uses the classical Artin-Hasse exponential to construct a ``splitting function'' for the $L$-functions of additive character sums over the torus $\G_m^d$. We will now use our notion of a flat lifting to construct general splitting functions over a smooth affine $\F_p$-scheme $X$.
		
		Let $R$ be a complete Noetherian local ring with maximal ideal $\frakm$ and residue field $\F_p$. Let $\chi:\Z_p \to R^\times$ be a continuous character. We associate to $\chi$ the parameter
		\begin{equation*}
			\pi_\chi	=	\chi(1)-1	\in \frakm.
		\end{equation*}
		Note that for any $a \in \Z_p$, we have $\chi(a) = (1+\pi_\chi)^a$ and thus $\chi \mapsto \pi_\chi$ gives a one-to-one correspondence between the set of $R$-valued characters of $\Z_p$ and $\frakm$. The functor sending $R \mapsto \frakm$ is representable, namely by the Iwasawa algebra $\Lambda = \Z_p\llbracket  T\rrbracket$. It follows that there is a universal \emph{$T$-adic character} $\chi_T:\Z_p \to \Lambda^\times$, which corresponds to the topologically nilpotent element $T \in \Lambda$.
		
		Let $X = \Spec(\overline{A})$ be an $\F_p$-scheme, and let $(A,\sigma)$ be a flat lifting of $X$ over $\Z_p$. For every character $\chi:\Z_p \to R^\times$ as above, we define a flat lifting $(A_{\pi_\chi},\sigma)$ over $R$, where
		\begin{equation*}
			A_{\pi_\chi} = R \otimes_{\Z_p} A,
		\end{equation*}
		and $\sigma:A_{\pi_\chi} \to A_{\pi_\chi}$ is the $R$-linear lifting of Frobenius obtained by base change. For every closed point $x \in |X|$, let us write $R(x) = R \otimes_{\Z_p} W(k(x))$. The Teichm\"uller map $\hat{x}$ induces a Frobenius-compatible map
		\begin{equation*}
			\hat{x}:A_{\pi_\chi}	\to	R(x).
		\end{equation*}
		Given $E \in A_{\pi_\chi}^\infty$, we will write $E(\hat{x}) \in R(x)$ for the ``value'' of $E$ at $\hat{x}$.
		
		\begin{definition}
			Let $f \in A^\infty$. A \emph{splitting function} for $f$ at $\chi$ is an element $E \in A_{\pi_\chi}^\infty$ 
			satisfying
			\begin{equation*}
				(1+\pi_\chi)^{\Tr_{k(x)/\F_p} f(\hat{x})}	= \Nm_{k(x)/\F_p} E(\hat{x})
			\end{equation*}
			for all $x \in |X|$. Here, we identify the Galois group of $R(x)/R$ with that of $k(x)/\F_p$.
		\end{definition}
		
		\begin{remark}
			Let $\rho:\pi_1(X) \to \Z_p$ be the map corresponding to $f$ by Artin-Schreier-Witt theory. Suppose that $R$ is a discrete valuation ring. In the terminology of \cite[\S 5.1]{Kramer-Miller-Upton}, a splitting function for $f$ at $\chi$ is a Frobenius structure for the unit-root $\sigma$-module over $A_{\pi_\chi}^\infty$ corresponding to $\rho_\chi = \chi \circ \rho$. If $f$ admits a splitting function at $\chi$, then this $\sigma$-module is free of rank $1$.
		\end{remark}
		
		Let us write $(A_T,\sigma)$ for the flat lifting corresponding to the $T$-adic character $\chi_T$. We claim that every $f \in A^\infty$ admits a splitting function $E_f(T) \in A_T^\infty$ at $\chi_T$. By specialization, we obtain a splitting function at every character $\chi:\Z_p \to R^\times$ as above. Recall that the \emph{Artin-Hasse exponential series} is the power series
		\begin{equation*}
			E(t)	=	\exp\left(	\sum_{i = 0}^\infty \frac{t^{p^i}}{p^i}	\right)	=	\prod_{p \nmid n} (1-t^n)^{-\mu(n)/n} \in 1+t+t^2\Z_p\llbracket  t\rrbracket.
		\end{equation*}
		Here, $\mu$ denotes the M\"obius function.
		
		\begin{lemma}[{\cite[Lemma 2.1]{KZ}}]\label{l:ahseries}
			For each $j \geq 0$, $E(t)$ induces a bijection $T^j\Lambda \to 1+T^j\Lambda$. Moreover, for all $r \in \Z_p$ and all $j \geq 0$ we have
			\begin{equation*}
				E(rT^j + T^{j+1}\Lambda) = 1+rT^j + T^{j+1}\Lambda.
			\end{equation*}
		\end{lemma}
		
		\begin{definition}
			For each $j \geq 0$, we let $\tau_j = \tau_j(T) \in T \Lambda$ be the unique series satisfying $E(\tau_j(T)) = (1+T)^{p^j}$.
		\end{definition}
		
		We are now ready to define our $T$-adic splitting functions. Recall from \cite[17.2.7]{Haze} that there is a functorial embedding of $\Z_p$-modules
		\begin{align*}
			W(A)	&\to	1+tA\llbracket t\rrbracket	\\
			a	&\mapsto	\prod_{j = 0}^\infty E(a_i t^{p^i}).
		\end{align*} 
		
		\begin{definition}
			For each $j \geq 0$, we define the $j$th \emph{Artin-Hasse exponential map} to be the composite
			\begin{equation*}
				p^j A		\xrightarrow{1/p^j}	A	\xrightarrow{D_\sigma}	W(A)	\to	1+tA\llbracket t\rrbracket	\xrightarrow{t \mapsto \tau_j}	1+\tau_j A\llbracket \tau_j\rrbracket \subseteq A_\Lambda^\infty.
			\end{equation*}
			
		\end{definition}
		
		Let us write $E_f^j = E_f^j(T)$ for the $j$th Artin-Hasse exponential of $f \in p^j A$. By construction, the assignment $f \mapsto E_f^j$ is $\Z_p$-linear and is functorial in the flat lifting $(A,\sigma)$.
			
		\begin{lemma}[Dwork's Splitting Lemma]\label{p:exp}
			Let $f \in p^j A^\infty$. Then $E_f^j(T)$ is a splitting function for $f$ at $\chi_T$.
		\end{lemma}
		\begin{proof}
			By functoriality, we have that $E_f^j(\hat{x}) = E_{f(\hat{x})}^j \in \Lambda(x)$. Since $k(x)$ is perfect, we have a Teichm\"uller expansion
			\begin{equation*}
				f(\hat{x})	=	p^j([f_0]+[f_1]p+\cdots)	\in W(k(x)).
			\end{equation*}
			By $\Z_p$-linearity of the Artin-Hasse exponential map, it suffices to prove the statement when $f(\hat{x}) = p^{j+k}[c]$. Note that since $\sigma([c]) = [c]^p$, we have $E_{p^j[c]}^j = E([c] \tau_j)$. But then
			\begin{equation*}
				(1+T)^{\Tr_{k(x)/\F_p} p^{j+k}[c]}	=	E(\tau_j)^{p^k \Tr_{k(x)/\mathbb{F}_p}[c]}	=	\Nm_{k(x)/\mathbb{F}_p} E([c]\tau_j)^{p^k}	=	\Nm_{k(x)/\mathbb{F}_p}	E_{p^{j+k}[c]}^j.
			\end{equation*}
		\end{proof}
		
\section{\texorpdfstring{$\Z_p$}{Zp}-Towers of Local Fields}\label{s:local}

	In this section we let $F$ be a local field of characteristic $p$ with residue field $\F_q$. Let $t \in F$ be a uniformizer. Consider the flat lifting $(\calA,\sigma)$, where $\calA = \Z_q(\! (t)\! )$ and $\sigma:\calA \to \calA$ is the unique lifting of Frobenius such that $\sigma(t) = t^p$. Let $G$ denote the absolute Galois group of $F$, and let $f \in \calA^\infty$. By Artin-Schreier-Witt theory, $f$ determines a map
	\begin{equation}\label{eq:localtower}
		\rho:G	\to	\Z_p.
	\end{equation}
	For each $j \geq 0$, the open subgroup $p^j \Z_p \subseteq \Z_p$ corresponds to a finite extension of local fields $F_j/F$. We abbreviate the resulting tower by $F_\infty/F$.
	
	\subsection{Local Towers with Strictly Stable Monodromy}\label{ss:strictlystable}
	
		Our first goal is to characterize the stable monodromy condition (\ref{d:striclystablefields}) in terms of a suitable representative $\tilde{f} \in \calA^\infty$ of the tower $F_\infty/F$.
		
		\begin{proposition}\label{p:localtwist}
			Let $\beta \in \F_q$ such that $\Tr_{\F_q/\F_p}(\beta) \neq 0$. There is a unique $\tilde{f} \in \calA^\infty$ such that $f \equiv \tilde{f} \pmod{\wp \calA^\infty}$ and
			\begin{equation*}
				\tilde{f}	=	c[\beta]	+	\sum_{k = 1}^\infty	c_kt^{-k},
			\end{equation*}
			where $c \in \Z_p$, $c_k \in \Z_q$ with $c_k = 0$ whenever $p | k$, and $c_k \to \infty$ as $k \to \infty$.
		\end{proposition}
		\begin{proof}
			Let $C$ be any basis for $\Z_q$ as a $\Z_p$-module, and let
			\begin{equation*}
				B	=	\{[\beta]\} \sqcup \{ct^{-k}:c \in C,p\nmid k\}.
			\end{equation*}
			The reduction of $B$ modulo $p$ forms a basis for $F/\wp F$ (see e.g. \cite[Example 2.4]{Kosters}). Since $\calA^\infty/\wp\calA^\infty$ is $p$-torsion free, it follows from \cite[Lemma 2.3]{Kramer-Miller-Upton} that the $p$-adic completion of the map
			\begin{align*}
				\bigoplus_{b \in B} \Z_p	&\to	\calA/\wp \calA	\\
				(c_b)_{b \in B}	&\mapsto	\sum_{b \in B} c_b b
			\end{align*}
			is an isomorphism.
		\end{proof}
		
		From now on we will fix a choice of $\beta$ as in Proposition \ref{p:localtwist} and let $c_0 = c[\beta]$. Note that the $p$-adic valuation of $c_0$ is independent of the choice of $\beta$. The $p$-adic growth of the coefficients $c_k$ is closely connected to the monodromy of $F_\infty/F$:
		
		\begin{proposition}[{\cite[Proposition 3.3]{Kosters}}]\label{p:vj}
			For all $j \geq 0$, let $v_j$ denote the highest ramification break (in upper numbering) of the finite extension $F_j/F$. Then
			\begin{equation*}
				v_j	=	\begin{cases}
						p^{j-1}\max\{	k p^{-v_p(c_k)}: v_p(c_k) < j	\}	&	\text{if }	v_p(c_k)	<	j	\text{ for some }	k	\\
						0	&	\text{otherwise}
					\end{cases}.
			\end{equation*}
		\end{proposition}
		
		It follows from Proposition \ref{p:vj} that $F_\infty/F$ has $\delta$-stable monodromy for $\delta \in \Z[\tfrac{1}{p}]$ if and only if the sequence $kp^{-v_p(c_k)}$ attains its maximum at some $k_0$, and the maximal value is equal to $\delta$. We will now give another characterization which is more useful for making explicit estimates: For each $k \geq 0$, consider the Teichm\"uller expansion
		\begin{equation}\label{eq:teich}
			c_k	=	\sum_{j=0}^\infty [c_{j,k}] p^j \in \mathbb{Z}_q.
		\end{equation}
		The convergence condition on the $c_k$ guarantees that we may write
		\begin{equation}\label{eq:hatf}
			\tilde{f}	=	\sum_{j = 0}^\infty	p^j F_j	=	\sum_{j = 0}^\infty p^j	\sum_{k = 0}^{d_j}	[c_{j,k}]t^{-k},
		\end{equation}
		where $F_j$ is a polynomial of degree $d_j$ in $t^{-1}$.
		
		\begin{lemma}\label{l:djpj}
			The following are equivalent:
			\begin{enumerate}
				\item	$F_\infty/F$ has $\delta$-stable monodromy.
				\item	The sequence $(d_j/p^j)$ has a maximum, with the maximal value equal to $\delta$.
			\end{enumerate}	
		\end{lemma}
		\begin{proof}
			Observe that if $v_p(c_k) > j$ then we must have $d_j < k$. In other words, the sequence $(d_j)$ is bounded above by the increasing sequence
			\begin{equation*}
				q_j = \max\{k:v_p(c_k) \leq j\}.
			\end{equation*}
			Moreover, we have $d_j = q_j$ whenever $j$ is the minimal index at which $d_j$ attains this value. It follows that $(d_j/p^j)$ has a maximum if and only if $(q_j/p^j)$ has a maximum, and that the two maxima agree. On the other hand, we have an upper bound
			\begin{align*}
				q_j	&=	p^j\max\{	kp^{-j}:v_p(c_k) < j+1	\}	\\
					&\leq	p^j\max\{	kp^{-v_p(c_k)}:v_p(c_k) < j+1	\}	\\
					&=	v_{j+1},
			\end{align*}
			with equality if and only if there exists some $k$ such that $v_p(c_k) = j$.
			
			First assume that $F_\infty/F$ has $\delta$-stable monodromy. Then the sequence $kp^{-v_p(c_k)}$ has maximal value $\delta$, so that $q_j \leq p^j \delta$ for all $j$. Let $k_0$ be the minimal value of $k$ such that $kp^{-v_p(c_k)} = \delta$, and let $j_0 = v_p(c_{k_0})$. Then $q_{j_0} = \delta p^{j_0}$ so that $q_j/p^j$ attains its maximal value $\delta$ when $j = j_0$.
			
			Conversely, suppose that $(d_j/p^j)$ has a maximum with maximal value $\delta$. Let $j_0$ be the smallest value of $j$ such that $q^j/p^j = \delta$, and let $k_0 = q_{j_0} = p^{j_0}\delta$. Then by the definition of $(q_j)$ we have $j_0 = v_p(c_{k_0})$, so that $q_{j_0} = v_{j_0+1}$. To complete the proof, it suffices to show that for all $k > k_0$ the value $kp^{-v_p(c_k)}$ is bounded by $\delta$. To see this, let $j = v_p(c_k)$. Since $k > k_0$, we have $j \geq j_0$ and since $v_p(c_k) < j+1$ we have
			\begin{equation*}
				k \leq q_j \leq p^j \delta.
			\end{equation*}
			Finally, we see that
			\begin{equation*}
				kp^{-v_p(c_k)} \leq p^j \delta p^{-v_p(c_k)} = \delta.
			\end{equation*}
		\end{proof}
		
		Our next goal is to construct a local splitting function with good $\pi_\chi$-adic growth properties. As in \ref{ss:split}, we first define a splitting function at the $T$-adic character $\chi_T$ and then specialize to obtain splitting functions at all other $\chi$.
		
		\begin{definition}
			The \emph{local splitting function} of the tower $F_\infty/F$ is defined to be
			\begin{equation}\label{eq:localsplitting}
				\tilde{E}(T)	=	\prod_{j = 0}^\infty	 E_{F_j}^j(T)	=	\prod_{j = 0}^\infty	\prod_{k = 0}^{d_j}	E([c_{j,k}]\tau_j(T) t^{-k})
			\end{equation}
		\end{definition}
		
		If $\chi$ is any continuous character of $\Z_p$, then we let $\tilde{E}(\pi_\chi) \in R$ denote the specialization of $\tilde{E}(T)$ along $T \mapsto \pi_\chi$. To study the $\pi_\chi$-adic convergence properties of $\tilde{E}(\pi_\chi)$, we recall our conventions on growth condtions from \cite[\S 2.1]{Kramer-Miller-Upton}. For every real number $m > 0$, consider the subring
		\begin{equation*}
			\calA_{\pi_\chi}^m	=	\left\{	\sum_{k = \infty}^\infty a_k t^{-k}:v_{\pi_\chi}(a_k) \geq \frac{k}{m}	\right\}.
		\end{equation*}
		Each $\calA_{\pi_\chi}^m$ is a $\pi_\chi$-adically complete subring of $\calA_{\pi_\chi}^\infty$. We define $\calA_{\pi_\chi}^\dagger$ to be the union of the $\calA_{\pi_\chi}^m$. Note in particular that $E_{F_j}^j(\pi_\chi) \in \calA_{\pi_\chi}^m$, where $m = d_j/v_{\pi_\chi}(\tau_j(\pi_\chi))$.
		
		\begin{lemma}\label{l:pij}
			Suppose that $\chi$ is finite or that $R$ has characteristic $p$. Then for all $j \geq 0$,
			\begin{equation*}
				v_{\pi_\chi}(\tau_j(\pi_\chi))	=	\begin{cases}
					p^j	&	0 \leq j < m_\chi	\\
					\infty	&	j \geq m_\chi
				\end{cases}.
			\end{equation*}
		\end{lemma}
		\begin{proof}
			By Lemma \ref{l:ahexp}, we have
			\begin{equation*}
				v_{\pi_\chi}(\tau_j(\pi_\chi)) = v_{\pi_\chi}(E(\tau_j(\pi_\chi))-1) = v_{\pi_\chi}((1+\pi_\chi)^{p^j}-1).
			\end{equation*}
			If $R$ has characteristic $p$, then this is equal to $v_{\pi_\chi}(\pi_\chi^{p^j}) = p^j$ for all $j$. Otherwise, the result follows since $1+\pi_\chi$ is a $p^{m_\chi}$-root of unity.
		\end{proof}
		
		As an immediate consequence of Lemma \ref{l:djpj} and Lemma \ref{l:pij}, we obtain an
		estimate fore our local splitting function:
		
		\begin{theorem}\label{t:local}
			Suppose that $F_\infty/F$ has $\delta$-stable monodromy. If $\chi$ is equicharacteristic or finite, then $\tilde{E}_{\pi_\chi} \in \calA_{\pi_\chi}^\delta$.
		\end{theorem}
		
	\subsection{Local Newton Polygons}\label{ss:localtoglobal}
	
		Let $\chi$ be a non-trivial $R$-valued character of $\Z_p$. Borrowing our terminology from \cite[\S 5.3]{Kramer-Miller-Upton}, we will say say that $\rho_\chi$ is \emph{overconvergent} if $\tilde{E}(\pi_\chi) \in \calA_{\pi_\chi}^\dagger$. More specifically, given a rational number $\delta > 0$ we will say that $\rho_\chi$ is \emph{$\delta$-overconvergent} if $\tilde{E}(\pi_\chi) \in \calA_{\pi_\chi}^\delta$. For example, Theorem \ref{t:local} states that if $F_\infty/F$ has $\delta$-stable monodromy, then $\rho_\chi$ is $\delta$-overconvergent if $\chi$ is finite or if $R$ has characteristic $p$.
		
		Let $U_p:\calA \to \calA$ be the local $U_p$-operator of \cite[\S 4.2]{Kramer-Miller-Upton}. If $\rho_\chi$ is overconvergent, then we define a local $p$-Dwork operator
		\begin{equation*}
			\tilde{\Theta}	=	U_p	\circ	\tilde{E}(\pi_\chi):A_{\pi_\chi}^\dagger \to A_{\pi_\chi}^\dagger,
		\end{equation*}
		The operator $\tilde{\Theta}$ is $R$-linear, and the iterate $\tilde{\Theta}_q = \tilde{\Theta}^{v_p(q)}$ is $R_q$-linear. The $R$-submodule
		\begin{equation*}
			\calA_{\pi_\chi}^{\dagger,\tr} = t^{-1}R_q\langle t^{-1}\rangle \cap \calA_{\pi_\chi}^\dagger
		\end{equation*}
		consisting of truncated overconvergent series in $t$ is invariant under the action of $\tilde{\Theta}$. By \cite[Theorem 2.1]{Monsky}, the action of $\tilde{\Theta}_q$ on the $K_q$-vector space $\calV_{\pi_\chi}^{\dagger,\tr} = K \otimes_R \calA_{\pi_\chi}^{\dagger,\tr}$ is nuclear.
		
		\begin{definition}
			The \emph{local normalized Newton polygon} of $\rho_\chi$ is defined to be
			\begin{equation*}
				\NP(\rho_\chi) = \NP_{\pi_{q,\chi}} (\tilde{\Theta}_q|\calV_{\pi_\chi}^{\dagger,\tr}).
			\end{equation*}
		\end{definition}
		
		\begin{definition}
			Let $\delta > 0$ be a rational number. We define the \emph{$\delta$-Hodge polygon} to be the infinite convex polygon with slope set
			\begin{equation*}
				\left\{	\frac{p-1}{\delta},\frac{2(p-1)}{\delta},...	\right\}.
			\end{equation*}
		\end{definition}
		
		\begin{proposition}
			If $\rho_\chi$ is $\delta$-overconvergent, then $\NP(\rho_\chi ) \succeq \HP(\delta)$.
		\end{proposition}
		
		\begin{theorem}\label{t:KZ}
			Suppose that $F_\infty/F$ is a $\Z_p$-tower of local fields with $\delta$-stable monodromy. Let $m_0$ be the smallest non-negative integer such that $v_j = p^{j-1} \delta$ for all $j > m_0$. Let $r > 0$. The following are equivalent:
			\begin{enumerate}
				\item	For some $\chi:\Z_p \to R^\times$ which is equicharacteristic or finite with $m_\chi > m_0$, the polygons $\NP^{< r}(\rho_\chi)$ and $\HP^{< r}(\delta)$ have the same terminal point.
				\item	For every $\chi:\Z_p \to R^\times$ which is equicharacteristic or finite with $m_\chi > m_0$, the polygons $\NP^{< r}(\rho_\chi)$ and $\HP^{< r}(\delta)$ have the same terminal point.
			\end{enumerate}
		\end{theorem}
	
		\begin{proof}
			This theorem more or less follows from the theory developed in \cite{KZ}. However, \cite{KZ} ignores equicharacteristic characters. Thus, we believe it
			is beneficial to outline their work and explain the extra observations needed to handle this case. In
			\cite{KZ}, Kosters and Zhu construct a ``very generic'' discrete valuation ring $R_\delta$ with valuation $v_\delta$ (note that they denote this ring by $R$). 
			For each $\chi$ there is an evaluation map $\mathrm{ev}_\chi: R_\delta \to R_\chi$. If $\chi$ is a finite character, then valuations
			``go up'' under the evaluation map
			\begin{align}\label{eq: kosters-zhu valuation obs 1}
				v_{\pi_\chi}(\mathrm{ev}_{\chi}(x)) &\geq v_\delta(x)
			\end{align}
			for all $x \in R_\delta$. 
			Furthermore, the following holds:
			\begin{align} \label{eq: kosters-zhu valuation obs 2}
				\begin{array}{l}
				\text{For some finite character } \\
				\text{$\chi$ with $m_{\chi}> m_0$ such that}\\
				\text{$v_{\pi_\chi}(\mathrm{ev}_\chi(x))= v_\delta(x)$ }
				\end{array}  &\iff \begin{array}{l}
				\text{For every finite character } \\
				\text{$\chi$ with $m_{\chi}> m_0$ we have}\\
				\text{$v_{\pi_\chi}(\mathrm{ev}_\chi(x))= v_\delta(x)$ }
				\end{array} 
			\end{align}
			This is Lemma 6.1 in \cite{KZ}. The proof of this Lemma extends without modification to include any character with values in a discrete valuation ring of characteristic $p$.

			Kosters and Zhu then construct a generic ``characteristic series'' $C(\pi,s) = 1+sR_\delta\llbracket s \rrbracket$ which satisfies a certain entireness property with respect to $v_\delta$, as well as the interpolation property
			\begin{equation*}
				\mathrm{ev}_\chi C(\pi,s) = C(\tilde{\Theta}_q|\calV_{\pi_\chi}^{\tr,\dagger},s)
			\end{equation*}
			for any non-trivial $\chi$. The entireness property of $C(\pi,s)$ allows us to define a Newton polygon $\NP C(\pi,s)$ using the normalized valuation $v_p(q) v_\delta$. The main estimate of Kosters-Zhu \cite[Proposition 5.7]{KZ} states that
			\begin{equation*}
				\NP C(\pi,s) \succeq \HP(\delta).
			\end{equation*}
			Now if $\chi$ is finite or $R$ has characteristic $p$, then \eqref{eq: kosters-zhu valuation obs 1} guarantees that $\NP(\rho_\chi) \succeq \NP C(\pi,s)$. The theorem follows immediately from \eqref{eq: kosters-zhu valuation obs 2}.
			
		\end{proof}
		
		\begin{corollary}\label{c:localperiodicity}
			In the setting of Theorem \ref{t:KZ}, suppose that $\chi$ equicharacteristic or finite with $m_\chi > m_0$. Let $d = p^{m_0}\delta$. For all $n \geq 1$, the polygons $\NP(\rho_\chi^\ext)$ and $\HP(\delta)$ agree on the interval $[nd-1,nd]$.
		\end{corollary}
		\begin{proof}
			By Theorem \ref{t:KZ}, it suffices to prove the statement for a finite character $\chi$ of order $m_\chi = m_0 + 1$. This is \cite[Theorem 7.10]{Kramer-Miller-Upton}.
		\end{proof}
		
		\begin{corollary}\label{c:deltaHodge}
			In the setting of Theorem \ref{t:KZ}, suppose that $\chi$ equicharacteristic or finite with $m_\chi > m_0$. Then $\rho_\chi$ is $\pi_\chi$-adically $\delta$-Hodge.
		\end{corollary}
		\begin{proof}
			The statement means that the $\delta$-Hodge polygon $\HP(\delta)$ agrees with a certain local Hodge polygon $\HP(\rho_\chi)$ that we have defined in \cite[\S 7.2]{Kramer-Miller-Upton}. Since the polygons $\NP(\rho_\chi)$ and $\HP(\delta)$ meet periodically by Corollary \ref{c:localperiodicity}, the statement follows immediately from the criterion \cite[Proposition 7.9]{Kramer-Miller-Upton}
		\end{proof}
		
		\begin{theorem}\label{t:localEquality}
			In the setting of Theorem \ref{t:KZ}, suppose that $\chi$ equicharacteristic or finite with $m_\chi > 0$.  Then $\NP(\rho_\chi)$ and $\HP(\delta)$ agree if and only if:
			\begin{enumerate}
				\item	\label{i:m0=0}$m_0 = 0$, so that $\delta \in \Z$.
				\item	\label{i:congruence}$p \equiv 1 \pmod{\delta}$.
			\end{enumerate}
		\end{theorem}
		\begin{proof}
			By Theorem \ref{t:KZ}, it suffices to prove the claim when $\chi$ is finite of order $p$. But we have shown this previously \cite[Theorem 1.3]{Kramer-Miller-Upton}.
		\end{proof}
		
		To conclude this section, let us briefly recall the notion of local-to-global extensions from \cite[\S 7.2]{Kramer-Miller-Upton}. Consider the pair $(A,\sigma)$, where $A = \Z_q[t^{-1}]$ and $\sigma:A \to A$ is the lifting of Frobenius defined by $\sigma(t^{-1}) = t^{-p}$. Since $\tilde{f} \in A^\infty$, the tower $F_\infty/F$ extends in a natural way to a tower $X_\infty/\A_{\F_q}^1$, or equivalently a continuous map
		\begin{equation*}
			\rho^\ext:\pi_1(\A_{\F_q}^1)	\to	\Z_p.
		\end{equation*}
		We may then consider the family of characters $\rho_\chi^\ext = \chi \circ \rho_\chi$, as $\chi$ varies through the continuous $R$-valued characters of $\Z_p$ as above. For each such $\chi$, let us write $L(\rho_\chi^\ext,s)$ for the Artin $L$-function of $\rho_\chi^\ext$ over $\A_{\F_q}^1$. If $\rho_\chi$ is overconvergent, then the Monsky trace formula guarantees that $L(\rho_\chi,s)$ is analytic in the disk $v_\pi(s) > -v_\pi(q)$. Moreover, we have the relation:
		\begin{equation*}
			\NP^{< e_\chi}(\rho_\chi^\ext) = \NP_{\pi_{q,\chi}}^{< e_\chi} L(\rho_\chi^\ext,s)	=	\NP^{< e_\chi}(\rho_\chi).
		\end{equation*}
		This shows that the truncated local polygon $\NP^{< e_\chi}(\rho_\chi)$ does not depend on our particular choice of lifting $(\calA,\sigma)$, or the choice of local splitting function $\tilde{E}(T)$.
		
\section{\texorpdfstring{$\Z_p$}{Zp}-Towers of Curves}

	In this final section we deduce our main theorems on $\Z_p$-towers of curves. Let $X/\F_q$ be a smooth affine curve whose smooth compatification $\overline{X}$ is ordinary. Let $X_\infty/X$ be a $\Z_p$-tower of curves with $\delta$-stable monodromy, corresponding to a continuous and surjective map
	\begin{equation*}
		\rho:\pi_1(X) \to \Z_p
	\end{equation*}
	
	
	\subsection{Results for Global Towers}\label{ss:slopeuniformity}
%
%
%
%
		For each $P \in S$, let $G_P$ denote the absolute Galois group of $F_P$. We obtain by restriction a $\Z_p$-tower of local fields $F_{P,\infty}/F_P$ with $\delta_P$-stable monodromy. This local tower corresponds to a continuous and surjective map
		\begin{equation*}
			\rho_P:G_P	\to	\Z_p.
		\end{equation*}
		As in  \S \ref{ss:localtoglobal}, the local character $\rho_P$ extends to $\rho_P^\ext:\pi_1(\A_{\F_q}^1) \to \Z_p$. Let $\chi$ be an $R$-valued character of $\Z_p$.
		Assume $\chi$ is finite or that $R$ has characteristic $p$. By Theorem \ref{t:local}, $\rho_\chi$ is $\pi_\chi$-adically \emph{$\delta$-Hodge} in the sense that $\rho_{\chi,P}$ is $\pi_\chi$-adically $\delta_P$-Hodge for every $P \in S$. By \cite[Proposition 4.11]{Kramer-Miller-Upton} and the Monsky trace formula \cite[Theorem 4.5]{Kramer-Miller-Upton}, the $L$-function $L(\rho_\chi,s)$ is entire.
		
		For each $j \geq 0$, let $v_{P,j}$ denote the highest ramification break (in upper numbering) of the finite extension $F_{P,j}/F_P$. We define $m_0$ to be the \emph{smallest} integer such that $v_{P,m_0+1} = p^{m_0} \delta_P$ for all $P$. Note that this means $v_{P,j} = p^{j-1} \delta_P$ for any $j> m_0$. We can now prove the combined form of Theorems \ref{t:globaltouching} and \ref{t:equicharacteristic}.

		\begin{theorem}
			Assume that $\overline{X}$ is ordinary. Let $d = p^{m_0}\sum_{P \in S} \delta_P$. Let $\chi$ be a character that is either finite or equicharacteristic and assume $m_\chi > m_0$. Then for all $1 \leq n < p^{m_\chi-m_0-1}$, the polygons $\NP^{< e_\chi}(\rho_\chi)$ and $\HP(X_\infty/X)$ agree on the interval
			\begin{equation}\label{eq:interval}
				[g-1+nd,g-1+|S|+nd].
			\end{equation}
		\end{theorem}
		\begin{proof}
			Let $e_0 = p^{m_0}(p-1)$, and let $\varepsilon < \min\{(p-1)/\delta_P:P \in S\}$. By Corollary \ref{c:localperiodicity}, the polygons $\NP^{< r}(\rho_{\chi,P}^\ext)$ and $\HP^{< r}(\delta_P)$ have the same terminal point for $r = n e_0$ or $ne_0-\varepsilon$. By \cite[Theorem 6.14]{Kramer-Miller-Upton}, the polygons $\NP^{< r}(\rho_\chi)$ and $\HP^{< r}(X_\infty/X)$ have the same terminal point for such $r$ as long as $r \leq e_\chi$. This condition is equivalent to $n < p^{m_\chi-m_0-1}$. The theorem follows since the restriction of $\HP(X_\infty/X)$ to the interval (\ref{eq:interval}) consists of $|S|$ segments of slope $n e_0$.
		\end{proof}
		
		\begin{proof}[Proof of Theorem \ref{theorem: intro stable implies uniformity}]
			Assume $m_\chi > m_0$. By renormalizing, we see that $\NP_q^{< 1}(\rho_\chi)$ lies above $\tfrac{1}{e_\chi}\HP^{<e_\chi}(\delta)$ and that both polygons meet periodically with period $d$. In particular, we see that $\NP_q^{< 1}(\rho_\chi) \to \tfrac{1}{e_\chi}\HP^{<e_\chi}(\delta)$ as $m_\chi \to \infty$. The result follows since the slopes of $\tfrac{1}{e_\chi}\HP^{<	e_\chi}(\delta)$ are equidistributed in $[0,1)$ as $m_\chi \to \infty$.
		\end{proof}
		
		\begin{theorem}\label{t:NPHPequal}
			If $\chi$ is finite (resp. the equicharacteristic character), then $\NP^{< e_\chi}(\rho_\chi^\ext)$ and $\HP^{< e_\chi}(X_\infty/X)$ (resp. $\NP(\rho_\chi^\ext)$ and $\HP(X_\infty/X)$) agree if and only if:
			\begin{enumerate}
				\item	$m_0 = 0$, so that $\delta_P \in \Z$ for all $P \in S$.
				\item	$p \equiv 1 \pmod{\delta_P}$ for all $P \in S$.
			\end{enumerate}
		\end{theorem}
		\begin{proof}
			Consider the finite character case. We know from \cite[Remark 1.4]{Kramer-Miller-Upton} that the conditions are necessary. By \cite[Theorem 6.14]{Kramer-Miller-Upton}, it suffices to prove that $\NP^{< e_\chi}(\rho_{\chi,P}^\ext)$ and $\HP^{< e_\chi}(\delta_P)$  agree for all $P \in S$. This follows immediately from Theorem \ref{t:localEquality}. The equicharacteristic case is almost identical.
		\end{proof}
		
		\begin{corollary}
			Let $\rho:\pi_1(X) \to \C_p^\times$ be a character of finite order $p^n$. For each $P \in S$, let $d_P$ denote the Swan conductor at $P$. Then $\NP^{< e_\chi}(\rho) = \HP^{< e_\chi}(\rho)$ if and only if
			\begin{enumerate}
				\item	$\overline{X}$ is ordinary.
				\item	$\delta_P = d_P/p^{n-1} \in \Z$ for all $P \in S$.
				\item	$p \equiv 1 \pmod{\delta_P}$ for all $P \in S$.
			\end{enumerate}			
		\end{corollary}
		\begin{proof}
			Once again, we know from \cite[Remark 1.4]{Kramer-Miller-Upton} that the conditions are necessary. Moreover, by \cite[Theorem 1.1]{Kramer-Miller-Upton}, it suffices to prove the claim when $X = \A_{\F_q}^1$. By Theorem \ref{t:NPHPequal}, we need only show that $\rho$ factors through a $\Z_p$-tower $X_\infty/\A_{\F_q}^1$ with $\delta_\infty$-stable monodromy at $\infty$. To see this, choose a factorization
			\begin{equation*}
				\rho:\pi_1(X)	\to	\Z/p^n\Z	\to	\C_p^\times.
			\end{equation*}
			The first map can be obtained via Artin-Schreier-Witt theory from a polynomial $\overline{f}  = \overline{c}_0+\cdots + \overline{c}_d t^d \in \Z_q/p^n\Z_q[t]$ with $\overline{c}_k = 0$ whenever $p | k$. Choose a lifting $f = c_0 + \cdots + c_d t^d \in \Z_q[t]$ of $\overline{f}$ such that $c_k = 0$ whenever $p | k$, and let $X_\infty/\A_{\F_q}^1$ be the corresponding tower. From Proposition \ref{p:vj}, we see immediately that $X_\infty/\A_{\F_q}^1$ has $\delta_\infty$-stable monodromy at $\infty$. This completes the proof.
		\end{proof}
		
	\subsection{Results for Twisted Towers}\label{ss:twisted}
		We are now able to prove Theorem \ref{t:twisted uniformity} and Theorem \ref{t:twisted stability} 
		by combining the results of \S \ref{ss:slopeuniformity} with the results of \cite{Kramer-Miller3}.
		We begin by recalling the notation and results from \cite{Kramer-Miller3}. Let $\psi:\pi_1(X) \to \Z_p^\times$
		be a finite tame character of order $c$. Note that $c|(p-1)$. Let $X^{tame}$ be the
		cover of $X$ corresponding to $\ker(\psi)$ and let $s_P$ be the ramification index over $P$.
		For each $P \in S$
		we obtain a local representation $\psi_P: G_P \to \Z_p^\times$. 
		There exists $e_P \in \frac{1}{p-1}\Z$ such that $G_P$ acts on $u_P^{e_P}$ by
		$\psi_P$. Note that $e_P$ is well defined up to addition by an integer. We define the \emph{exponent}
		$\mathbf{e}_P \in \frac{1}{p-1}\Z/\Z$ of $\psi$ at $P$ to be the equivalence class of $e_P$ modulo $\Z$.
		This is analogous to the notion of exponents in the theory of complex differential equations with regular 
		singularities. We then define $0\leq \epsilon_P\leq p-2$ to be the unique integer such that $\frac{\epsilon_P}{p-1}$
		is in $\mathbf{e}_P$. Note that $\epsilon_P=0$ if and only if $\psi$ is unramified at $P$.
		Finally, we define a global invariant
		\begin{align*}
			\Omega_{\psi}&= \frac{1}{p-1}\sum\limits_{P\in S} \epsilon_P.
		\end{align*}
		One may show that $\Omega_\psi$ is an integer (see \cite[\S 5.3.2]{Kramer-Miller3}). 
		
		Let $X_\infty/X$ be $\delta$-stable tower. We define the ``twisted'' Hodge polygon as follows:
		\begin{align*}
			\HP(\psi \otimes X_\infty/X) &= \underbrace{\{0,\dots,0\}}_{g-1 + |S| - \Omega_\psi}\sqcup \bigsqcup_{P \in S} 
			\Bigg\{\frac{1-\epsilon_P}{\delta_P},\frac{2-\epsilon_P}{\delta_P}, \dots \Bigg \} . 
		\end{align*}
		By \cite[Theorem 1.1]{Kramer-Miller3} we have 
		\begin{align}\label{eq:twisted NP over HP}
			\NP^{<e_\chi}(\psi \otimes \chi) \succeq \HP^{<e_\chi}(\psi \otimes X_\infty/X)
		\end{align}
		for any finite character $\chi:\Gal(X_\infty/X) \to \C_p^\times$. 
		
		\begin{proof}
			(Proof of Theorem \ref{t:twisted uniformity} and Theorem \ref{t:twisted stability}) Let $X^{\tame}_\infty/X^{\tame}$ be the pullback
			of the tower $X_\infty/X$ along $X^{\tame} \to X$. Note that $X^{\tame}_\infty/X^{\tame}$ is
			$\delta^{\tame}$-stable, where $\delta^{\tame}_P=s_P\delta_P$. We have a decomposition
			\begin{align*}
				\mathrm{Ind}^{X^{\tame}}_X(X_\infty^{\tame}/X^{\tame}) &= \bigoplus_{i=0}^{c-1} \psi^{\otimes i} \otimes X_\infty/X.
			\end{align*}
			This follows by considering the induction $\mathrm{Ind}^{X^{\tame}}_{X}(\chi^{\tame})$ for any character $\chi^{\tame}:\Gal(X_\infty^{\tame}/X^{\tame}) \to \C_p^\times$. We obtain a corresponding decomposition of Hodge polygons
			\begin{align}\label{eq:decomposition of Hodge polygons}
				\HP(X^{\tame}_\infty/X^{\tame}) &= \bigsqcup_{i=0}^{c-1} \HP(\psi^{\otimes i} \otimes X_\infty/X).
			\end{align}
			This decomposition essentially follows from the Riemann-Hurwitz theorem and by observing how
			the exponents of $\psi^{\otimes i}$ vary with $i$. Similarly, for any $\chi:\Gal(X_\infty/X) \to \C_p^\times$
			we have a decomposition
			\begin{align}\label{eq:decomposition of Newton polygons}
				\NP(\chi^\tame)&=\bigsqcup_{i=0}^{c-1} \NP(\psi^{\otimes i} \otimes \chi).
			\end{align}
			By combining equations \eqref{eq:decomposition of Hodge polygons} and \eqref{eq:decomposition of Newton polygons}
			with the Newton-over-Hodge result \eqref{eq:twisted NP over HP}, we see that $\NP^{<r}(\chi^\tame)$ and $\HP^{<r}(X^{\tame}_\infty/X^{\tame})$ have the same endpoints if and only if 
			$\NP^{<r}(\psi^{\otimes i} \otimes \chi)$ and $\HP^{<r}(\psi^{\otimes i} \otimes X_\infty/X)$
			have the same endpoints for each $0\leq i < c$. Slope uniformity then follows from Theorem \ref{t:globaltouching}
			applied to $X^{\tame}_\infty/X^{\tame}$. The slope stability result follows from Theorem \ref{theorem: intro stable with congruence implies stability} applied to $X^{\tame}_\infty/X^{\tame}$.
		\end{proof}
%
%
%
		
\bibliographystyle{plain}
\bibliography{Stable}
\end{document}